\documentclass[11pt]{amsart}

\usepackage{mathrsfs}
\usepackage{amssymb,amsmath,amsxtra,amscd, feynmf,latexsym}
\usepackage{pst-all}
\usepackage{mathpple}
\usepackage{setspace,ifpdf}
\usepackage{diagrams}

\usepackage{pst-tree, pst-node}
\usepackage[mathscr]{eucal} 

\ifpdf
\usepackage{pdfsync}
\fi

\theoremstyle{plain}
\numberwithin{equation}{section}
\newtheorem{theorem}{Theorem}

\newtheorem{lemma}{Lemma}
\numberwithin{lemma}{section}
\numberwithin{corollary}{section}
\numberwithin{proposition}{section}

\theoremstyle{definition}
\newtheorem{definition}{Definition}
\numberwithin{definition}{section}

\theoremstyle{plain}
\newtheorem{example}{Example}
\numberwithin{example}{section}

\newcommand{\nc}{\newcommand}
\nc{\C}{\mathcal{C}}
\nc{\CC}{\widetilde{C}}
\nc{\wt}{\overline}
\nc{\mc}{\mathcal}
\nc{\on}{\operatorname}
\nc{\Cl}{Cl}
\nc{\drva}{\widehat{\OOmega}}
\nc{\spec}{\OOn{Spec}}
\nc{\AutO}{\OOn{Aut} \mc{O}}
\nc{\vac}{|0\rangle}
\nc{\Z}{\mathbb{Z}}
\nc{\zf}[1]{z^{\frac{1}{#1}}}
\nc{\wf}[1]{w^{\frac{1}{#1}}}
\nc{\Mt}{M^{\sigma}}

\nc{\gr}{\OOn{gr}}
\nc{\T}{\mathbb{T}}
\nc{\LT}{\mathbb{LT}}
\nc{\CT}{\mathbb{C}\{ \mathbb{T} \}}
\nc{\g}{\mathfrak{g}}
\nc{\A}{\mathcal{A}}
\nc{\n}{\mathfrak{n}}
\nc{\I}{\mathcal{I}}
\renewcommand{\H}{\on{H}}
\nc{\Hc}{\mathcal{H}}
\nc{\HH}{\mathbf{H}}
\nc{\M}{M}
\nc{\U}{\mathcal{U}}
\nc{\F}{\mathcal{F}}
\nc{\OO}{\mathcal{O}}
\nc{\LRF}{\mathcal{LRF}}
\nc{\CRF}{\mathcal{CRF}}
\nc{\FD}{\mathcal{LFG}}
\nc{\LFD}{\mathcal{LFG}}
\nc{\FG}{\mathcal{LFG}}
\nc{\LFG}{\mathcal{LFG}}
\nc{\labb}{\emph{lab}}

\nc{\NC}{\on{NC}}
\nc{\Q}{\on{QSym}}

\nc{\Iso}{\on{Iso}}

\nc{\heck}{\mc{H}ecke}
\renewcommand{\F}{\mathcal{F}}

\nc{\al}{\alpha}
\nc{\OOl}{\OOverline}

\begin{document}

\title{Colored trees and noncommutative symmetric functions}
\author{Matt Szczesny}
\address{Department of Mathematics  
         Boston University, Boston MA, USA}
\email{szczesny@math.bu.edu}

\begin{abstract}

Let $\CRF_S$ denote the category of $S$-colored rooted forests, and $\H_{\CRF_S}$ denote its Ringel-Hall algebra as introduced in \cite{KS}. 
We construct a homomorphism from a $K^+_0 (\CRF_S)$--graded version of the Hopf algebra of noncommutative symmetric functions to $\H_{\CRF_S}$. Dualizing, we obtain a homomorphism from the Connes-Kreimer Hopf algebra to a $K^+_0 (\CRF_S)$--graded version of the algebra of quasisymmetric functions. This homomorphism is a refinement of one considered by W. Zhao in \cite{Z}.  
\end{abstract}

\maketitle 

\section{Introduction}

In \cite{KS} categories $\LRF, \LFG$ of labeled rooted forests and labeled Feynman graphs where constructed, and were shown to possess many features in common with those of finitary abelian categories. In particular, one can define the Ringel-Hall algebras $\H_{\LRF}, \H_{\LFG}$ of the categories $\LRF, \LFG$. If $\C$ is one of these categories, $\H_{\C}$ is the algebra of functions on isomorphism classes of $\C$, equipped with the convolution product
\begin{equation} \label{conv_product}
f \star g (M) = \sum_{A \subset M} f(A) g(M/A)
\end{equation}
and the coproduct
\begin{equation} \label{coprod}
\Delta(f) (M, N) := f(M \oplus N)
\end{equation}
where $M \oplus N$ denotes disjoint union of forests/graphs. Together, the structures \ref{conv_product} and \ref{coprod} assemble to form a co-commutative Hopf algebra, which was in \cite{KS} shown to be  dual to the corresponding Connes-Kreimer Hopf algebra (\cite{K}, \cite{CK}). In \cite{KS}, we also defined the Grothendieck groups $K_0 (\C)$ for $\C=\LRF, \LFG$  (which is not trivial, since $\C$ is not abelian), and showed that $\H_{\C}$ is naturally graded by $K^+_0 (\C)$ - the effective cone inside $K_0 (\C)$. 

From the point of view of Ringel-Hall algebras of finitary abelian categories, the characteristic functions of classes in  $K^+_0$ are interesting. If $\A$ is such a category, and $\alpha \in K^+_0 (\A)$,  we may consider $\kappa_{\alpha}$ -  the characteristic function of the locus of objects of class $\alpha$ inside $\on{Iso}(\A)$ (for a precise definition, see \cite{J}). It is shown there that the $\kappa_{\alpha}$ satisfy
\begin{equation} \label{div_power}
\Delta(\kappa_{\alpha}) = \sum_{\substack{\alpha_1 + \alpha_2 = \alpha \\ \alpha_1, \alpha_2 \in K^+_0 (\A) }} \kappa_{\alpha_1} \otimes \kappa_{\alpha_2}
\end{equation}
In this note, we show that these identities hold also when $\A$ is replaced by the category $\CRF$ of \emph{colored rooted forests}. If $S$ is a set, and $\CRF_S$ denotes the category of rooted forests colored by $S$, we show that $K_0 (\CRF_S)= \mathbb{Z}^{|S|}$, and if $\alpha \in K^+_0 (\CRF_S)$, we may define
\[
\kappa_{\alpha} := \sum_{ \substack{ \{ A \in \on{Iso}(\CRF_S), \\ [A] = \alpha \}}} \delta_{A}
\]
i.e. the sum of delta functions supported on isomorphism classes with K-class $\alpha$. We show that the $\kappa_{\alpha}$ satisfy the identity \ref{div_power}. 

As an application, we construct a homomorphism to $\H_{\CRF_S}$ from a $K^+_0(\CRF_S)$--graded version of the Hopf algebra of non-commutative symmetric functions (see \cite{GKLLRT}). More precisely, let $\NC_{\CRF_S}$ denote the free associative algebra on generators $X_{\alpha}, \alpha \in K^+_0 (\CRF_S)$, to which we assign degree $\alpha$. We may equip it with a coproduct determined by the requirement
\[
\Delta(X_\alpha) = \sum_{\substack{\alpha_1 + \alpha_2 = \alpha \\ \alpha_1, \alpha_2 \in K^+_0 (\CRF_S) }} X_{\alpha_1} \otimes X_{\alpha_2}
\]
with which it becomes a connected graded bialgebra, and hence a Hopf algebra. We may now define a homomorphism
\begin{align*}
\rho: \NC_{\CRF_S} & \rightarrow \H_{\CRF_S} \\
\rho(X_{\alpha}) = \kappa_{\alpha}
\end{align*}
This is a refinement of a homomorphism originally considered in \cite{Z}. 
Taking the transpose of $\rho$, we obtain a homomorphism from the Connes-Kreimer Hopf algebra to a $K^+_0(\CRF_S)$--graded version of the Hopf algebra of quasisymmetric functions. 
\bigskip

\noindent {\bf Acknowledgements:} I would like to thank Dirk Kreimer for many valuable conversations.   

\section{Recollections on $\CRF_S$}

We briefly recall the definition and necessary properties of the category $\CRF_S$, and  calculate its Grothendieck group. For details and proofs, see \cite{KS}. While \cite{KS} treats the case of uncolored trees, the extension of the results to the colored case is immediate. Please note that the notion of labeling in \cite{KS} and coloring used here are distinct. 

\subsection{The category $\CRF_S$}

\label{LCRF}
We begin by reviewing some notions related to rooted trees. Let $S$ be a set. For a tree $T$, denote by $V(T), E(T)$ the vertex and edge sets of $T$ respectively.   

\begin{definition}
\begin{enumerate}
\item A \emph{rooted tree colored by $S$} is a tree $T$, with a distinguished vertex $r(T)\in V(T)$ called the root, and an map $l: V(T) \rightarrow S$. An isomorphism between two trees $T_1, T_2$ labeled by $S$ is a pair of bijections $f_v: V(T_1) \simeq V(T_2)$, $f_e: E(T_1) \simeq E(T_2)$ which preserve roots, colors, and all incidences - we often refer to this data simply by $f$. Denote by $RT(S)$ the set of all rooted trees labeled by $S$. 
\item A \emph{rooted forest colored $S$} is either empty, or an ordered set $F = \{T_1, \cdots, T_n \}$ where $T_i \in RT(S)$. $F_1 = \{ T_1, \cdots, T_n \}$ and $F_2 =\{ T'_1, \cdots, T'_m \}$ are isomorphic if $m=n$ and there is a permutation $\sigma \in S_n $, together with isomorphisms $f_i: T_i \simeq T'_{\sigma(i)}$. 
\item An \emph{admissible cut} of a labeled colored tree $T$ is a subset $C(T) \subset E(T)$ such that at most one member of $C(T)$ is encountered along any path joining a leaf to the root. Removing the edges in an admissible cut divides $T$ into a colored rooted forest $P_C(T)$ and a colored rooted tree $R_C(T)$, where the latter is the component containing the root. The \emph{empty} and \emph{full} cuts $C_{null}, C_{full}$, where $$(P_{C_{null}}(T), R_{C_{null}}(T)) = (\emptyset, T) \textrm{ and }  (P_{C_{full}}(T), R_{C_{full}}(T)) = (T, \emptyset)$$ respectively, are considered admissible.
\item An \emph{admissible cut} on a colored forest $F = \{ T_1, \cdots, T_k \}$ is a collection of cuts  $C=\{C_1, \cdots, C_k \}$, with $C_i$ an admissible cut on $T_i$. Let
\begin{align*}
R_C (F) & := \{ R_{C_1}(T_1), \cdots, R_{C_k}(T_k) \} \\
P_C (F) & := P_{C_1}(T_1) \cup P_{C_2} (T_2) \cup \cdots \cup P_{C_k} (T_k)
\end{align*}

\end{enumerate}
\end{definition}

\begin{example} Consider the labeled rooted forest consisting of a single tree $T$ colored by $S=\{ a, b \}$ with root drawn at the top. 

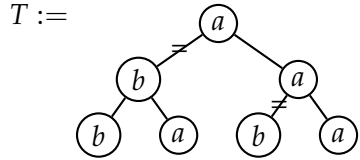
\begin{center} \psset{levelsep=4ex, treesep=0.5cm}
$T:=$ \pstree{  \Tcircle{a}  } {  \pstree{ \Tcircle{b} \ncput{=}   } { \Tcircle{b} \Tcircle{a}}
 \pstree{\Tcircle{a}}{\Tcircle{b}\ncput{=} \Tcircle{a}}   }
\end{center}
and the cut edges are indicated with "=", then 
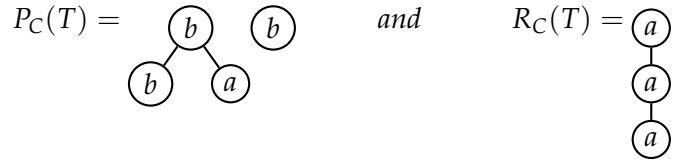
\begin{center}  \psset{levelsep=4ex, treesep=0.5cm}
$P_C(T) = $ \pstree{ \Tcircle{b}}{\Tcircle{b} \Tcircle{a}} \hspace{3cm} \Tcircle{b} \hspace{1cm}
and
\hspace{1cm}
$R_C(T) = $ \pstree{\Tcircle{a}} {\pstree{\Tcircle{a}} {\Tcircle{a}} }
\end{center}

\end{example}

\noindent We are now ready to define the category $\CRF_S$, of rooted forests colored by $S$. 

\begin{definition} The category $\CRF_S$ is defined as follows:
\begin{itemize}
\item
 \[
\on{Ob}(\CRF_S) = \{ \textrm{ rooted forests } F \textrm{ colored by S } \} 
\]
\item 
\begin{align*}
\on{Hom}(F_1,F_2) := & \{ (C_1,C_2,f) | C_i \textrm{ is an admissible cut of } F_i, \\ & \; f: R_{C_1} (F_1) \cong P_{C_2} (F_2) \} \; \; F_i \in \on{Ob}(\CRF_S). 
\end{align*}
{\bf Note:} For $F \in \CRF_S$, $(C_{null}, C_{full}, id): F \rightarrow F$ is the identity morphism in $\on{Hom}(F,F)$. We denote by $\on{Iso}(\CRF_S)$ the set of isomorphism classes of objects in $\CRF_S$. 
\end{itemize}

\end{definition}

\newpage

\noindent {\bf Example:} if 
\bigskip
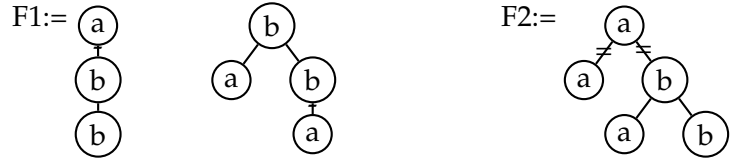
\begin{center} \psset{levelsep=4ex, treesep=0.5cm}
F1:= \pstree[]{ \Tcircle{a} }{ \pstree{ \Tcircle{b} \ncput{-} } {\Tcircle{b} } } \hspace{1 cm} 
\pstree{\Tcircle{b}}{\Tcircle{a} {\pstree{\Tcircle{b}}{\Tcircle{a}\ncput{-} } }} \hspace{2cm}
F2:= \pstree{\Tcircle{a}}{ \Tcircle{a}\ncput{=} \pstree{\Tcircle{b}\ncput{=} }{\Tcircle{a}\Tcircle{b}}}
\end{center}
\bigskip
then we have a morphism $(C_1, C_2, f)$  where $C_i$ are as indicated in the diagram, and $f$ is uniquely determined by the cuts. 

\bigskip

For the definition of composition of morphisms and a proof why it is associative, please see \cite{KS}. The category $\CRF_S$ has several nice properties:

\begin{enumerate}
\item The empty forest $\emptyset$ is a null object in $\CRF_S$.
\item Disjoint union of forests equips $\CRF_S$ with a symmetric monoidal structure. We denote by $F_1 \oplus F_2$ the disjoint union of the rooted forests $F_1$ and $F_2$ labeled by $S$, and refer to this as the direct sum. 
\item Every morphism possesses a kernel and a cokernel.
\item For every admissible cut $C$ on a forest $F$, we have the short exact sequence
\begin{equation} \label{ses}
\emptyset \rightarrow P_C(F) \overset{(C_{null},C, id)}{\longrightarrow} F \overset {(C,C_{full},id)}{\longrightarrow} R_C  (F) \rightarrow \emptyset
\end{equation} 
\end{enumerate}

The second property above allows us to define the Grothendieck group of $\CRF_S$ as
\[
K_0 (\CRF_S) := \mathbb{Z}[\on{Iso}(\CRF_S] / \sim
\]
i.e. the free abelian group generated by isomorphism classes of objects modulo the relation $\sim$, where $\sim$ is generated by differences $B-A-C$ for short exact sequences
\[
\emptyset \rightarrow A \rightarrow B \rightarrow C \rightarrow \emptyset.
\] 
We denote by $[A]$ the class of $A \in \CRF_S$ in $K_0 (\CRF_S)$.

\begin{lemma}
$K_0 (\CRF_S) \simeq \mathbb{Z}^{\oplus |S|}$
\end{lemma}

\begin{proof}
For a rooted forest, let $v(F,s)$ denote the number of vertices in $F$ of color $s \in S$. Let $\mathbb{Z}^{S}$ denote the free abelian group on the set $S$, with generators  $e_s, s \in S$. Let
\begin{align*}
\Psi: \mathbb{Z}[\on{Iso}(\CRF_S)] & \rightarrow \mathbb{Z}^S \\
\Psi(F) &= \sum_{s \in S} v(F,s) e_s
\end{align*}
The subgroup generated by the relations $\sim$ like in the kernel of $\Psi$, so we get a well-defined group homomorphism
\[
\overline{\Psi}: K_0 (\CRF_S) \rightarrow \mathbb{Z}^S
\]
Now, let 
\begin{align*}
\Phi: \mathbb{Z}^S & \rightarrow K_0 (\CRF_S) \\
\Phi(\sum_{s} a_s e_s) & = \sum_{s \in S} a_s [\bullet_s ]
\end{align*}
where $\bullet_s$ denotes the Forest with one vertex colored $s$. We see that $\overline{\Psi}$ and $\Phi$ are inverse to each other, and so the result follows. 

\end{proof}

We denote by $K^+_0 (\CRF_S) \simeq \mathbb{N}^{|S|}$ the cone of effective classes in $K_0 (\CRF_S)$.

\section{Ringel-Hall algebras}
We recall the definition of the Ringel-Hall algebra of $\CRF_S$ following \cite{KS}. For an introduction to Ringel-Hall algebras in the context of abelian categories, see \cite{S}. 
We define the Ringel-Hall algebra of $\CRF_S$, denoted $\H_{\CRF_S}$, to be the $\mathbb{Q}$--vector space of finitely supported functions on isomorphism classes of $\CRF_S$. I.e.
\[
\H_{\CRF_S} := \{ f: \on{Iso}(\CRF_S) \rightarrow \mathbb{Q} | |supp(f)| < \infty  \}
\]
As a $\mathbb{Q}$--vector space it is spanned by the delta functions $\delta_A, A \in \on{Iso}(\CRF_S)$. The algebra structure on $\H_{\CRF_S}$ is given by the convolution product:
\[
f \star g (M) = \sum_{A \subset M} f(A) g(M/A) 
\]
$\H_{\CRF_S}$ possesses a co-commuative co-product given by
\begin{equation} \label{cop}
\Delta(f)(M,N)=f(M \oplus N) 
\end{equation}
as well as a natural $K^+_0 (\CRF_S)$--grading in which  $\delta_A$ has degree $[A] \in K^+_0 (\CRF_S)$. The algebra and co-algebra structures are sompatible, and $\H_{\CRF_S}$ is in fact a Hopf algebra (see \cite{KS}). It follows from \ref{cop} that
\begin{equation}
\Delta(\delta_A) = \sum_{\{ A', A'' | A' \oplus A'' \simeq A \}} \delta_{A'} \otimes \delta_{A''}
\end{equation}
where the sum is taken over all \emph{distinct} ways of writing $A$ as $A' \oplus A''$ . 

\section{$K^+_0 (\CRF_S)$--graded noncommutative symmetric functions and homomorphisms}

Let $\NC_{\CRF_S}$ denote the free associative algebra on $K^+_0 (\CRF_S)$, i.e. the free algebra generated by variables $X_{\alpha}, \alpha \in K^+_0 (\CRF_S)$. We give it the structure of a Hopf algebra through the coproduct
\begin{equation} \label{free_cop}
\Delta( X_{\gamma}) = \sum_{ \substack{ \{ \alpha + \beta = \gamma \\ \alpha, \beta \in K^+_0 (\CRF_S)\}}} X_{\alpha} \otimes X_{\beta} 
\end{equation}
and equip it with a $K^+_0 (\CRF_S)$--grading by assigning $X_{\alpha}$ degree $\alpha$. For $\alpha \in K^+_0 (\CRF_S)$, let
\[
\kappa_{\alpha} = \sum_{A \in \on{Iso}(\C), [A] = \alpha} \delta_{A}
\]
This is a $K^+_0 (\CRF_S)$--graded version of the Hopf algebra of non-commutative symmetric functions (see \cite{GKLLRT}).

\begin{example} \label{ex1}
Suppose that $S = \{ a,b \}$. We then have $K_0 (\CRF_S) \simeq \mathbb{Z}^2$, and may identify the pair $(i,j) \in K^+_0 (\CRF_S)$ as the class representing forests possessing $i$ vertices colored $"a"$ and $j$ colored $"b"$. We have for instance
\[
\kappa_{(1,1)} = \psset{levelsep=2ex, treesep=0.6cm, treenodesize=1pt,nodesepB=-2pt}
\delta_{  \pstree{\Tc*{3pt}~[tnpos=r]{a}}{\Tc*{3pt}~[tnpos=r]{b}}} + 
 \delta_{ \pstree{\Tc*{3pt}~[tnpos=r]{b}}{\Tc*{3pt}~[tnpos=r]{a}}} +
 \delta_{\Tc*{3pt}~[tnpos=b]{a} \oplus \Tc*{3pt}~[tnpos=b]{b}}  
\]   
\end{example}

\begin{theorem}
The map $\rho: \NC_{\CRF_S} \rightarrow \H_{\CRF_S}$ determined by $\rho(X_{\alpha}) =\kappa_{\alpha}$ is a Hopf algebra homomorphism. 
\end{theorem}

\begin{proof}
Since $\NC_{\CRF_S}$ is free as an algebra, we only need to check that the $\kappa_{\alpha}$ are compatible with the coproducts \ref{free_cop}, i.e. that
\begin{equation} \label{kappa_cop}
\Delta(\kappa_{\gamma}) = \sum_{ \substack{\alpha + \beta = \gamma \\ \alpha, \beta \in K^+_0 (\CRF_S)}} \kappa_{\alpha} \otimes \kappa_{\beta} 
\end{equation}
We have
\begin{align*}
\Delta(\kappa_{\gamma}) & = \sum_{ \{A \in \on{Iso}(\C) | [A] = \gamma \}} \Delta(\delta_{A}) \\
& = \sum_{ \{A \in \on{Iso}(\CRF_S) | [A] = \gamma \}} \sum_{\{A', A'' | A' \oplus A'' \simeq A\}} \delta_{A'} \otimes \delta_{A''}
\end{align*}
the result now follows by observing that the term $\delta_{A'} \otimes \delta_{A''}$ occurs exactly once in  $\kappa_{[A']} \otimes \kappa_{[A'']}$, which is an element of the right-hand side of \ref{kappa_cop}, since $[A']+[A'']=\gamma$. 

\end{proof}

\subsection{Connection to work of W. Zhao}

Let $\on{NC}$ denote the "usual" Hopf algebra of non-commutative symmetric functions. I.e. $\on{NC}$ is the free algebra on generators $Y_n,  n \in \mathbb{N}$, with coproduct defined by
\[
\Delta(Y_n) = \sum_{i+j =n} Y_{i} \otimes Y_{j}
\]
(we adopt the convention that $Y_0=1$). Suppose that the labeling set $S$ is a subset of $\mathbb{N}$. We then have group homomorphism 
\begin{align*}
V:K_0 (\CRF_S) & \rightarrow \mathbb{N}  \\
V(\sum a_s e_s) & := \sum a_s s
\end{align*}
which simply amounts to adding up the labels in a given forest. We can now define an algebra homomorphism 
\begin{align*}
J_S : \on{NC} & \rightarrow \NC_{\CRF_S} \\
J_S (Y_n) & := \sum_{ \substack{ \{ \alpha \in K^+_0 (\CRF_S) | \\ V(\alpha) = n \}}} X_{\alpha}
\end{align*}

\begin{lemma}
$J_S$ is a Hopf algebra homomorphism
\end{lemma}

\begin{proof}
We only need to check the compatibility of the coproduct. We have
\begin{align*}
\Delta(J_S (Y_n)) &=  \sum_{ \substack{ \{ \alpha \in K^+_0 (\CRF_S) | \\ V(\alpha) = n \}}} \Delta(X_{\alpha}) \\
&=  \sum_{ \substack{ \{ \alpha \in K^+_0 (\CRF_S) | \\ V(\alpha) = n \}}} \sum_{\gamma_1 + \gamma_2 = \alpha} X_{\gamma_1} \otimes X_{\gamma_2} \\
&= \sum_{\{ \gamma, \gamma' | V(\gamma) + V(\gamma') = n } X_{\gamma} \otimes X_{\gamma'} \\
& = J_S (\Delta (Y_n))
\end{align*}
\end{proof}

Composing $\rho$ and $J_S$, we obtain a Hopf algebra homomorphism
\begin{align*}
\rho \circ J_S : \on{NC} & \rightarrow \H_{\CRF_S} \\
\rho \circ J_S (Y_n) & = \sum_{\substack{A \in \on{Iso}(\CRF_S), \\ V([A]) = n}} \delta_{A} 
\end{align*}
which was considered in \cite{Z}. 

\section{The transpose of $\rho$}

The graded dual of the Hopf algebra $\NC_{\CRF_S}$ is a $K^+_0 (\C)$--graded version of the Hopf algebra of quasi-symmetric functions (see \cite{C}), which we proceed to describe. Let $\Q_{\CRF_S}$ denote the $\mathbb{Q}$--vector space spanned by the symbols $Z(\alpha_1, \alpha_2, \cdots, \alpha_k),  k \in \mathbb{N}, \alpha_i \in K^+_{0}(\CRF_S)$. We make $\Q_{\CRF_S}$ into a co-algebra via the coproduct 
\begin{align*}
\Delta(Z(\alpha_1, \cdots, \alpha_k)) &= 1 \otimes Z(\alpha_1, \cdots, \alpha_k)  \\
& + \sum^{k-1}_{i=1} Z(\alpha_1, \cdots, \alpha_i) \otimes Z(\alpha_{i+1}, \cdots, \alpha_{k}) + Z(\alpha_1, \cdots, \alpha_k) \otimes 1
\end{align*}
The algebra structure on $\Q_{\CRF_S}$ is given by the \emph{quasi-shuffle} product, as follows. Given $Z(\alpha_1, \cdots, \alpha_k)$ and $Z(\beta_1, \cdots, \beta_l)$, their product is determined by: 
\begin{enumerate}
\item Inserting zeros into the sequences $\alpha_1, \cdots, \alpha_k$ and $\beta_1, \cdots, \beta_l$ to obtain two sequences $\nu_1, \cdots, \nu_p$ and $\mu_1, \cdots, \mu_p$ of the same length, subject to the condition that for no $i$ do we have $\nu_i=\mu_i=0$.
\item For each such pair $\nu_1, \cdots. \nu_p$, and $\mu_1, \cdots, \mu_p$, writing $Z(\nu_1 + \mu_1, \cdots, \nu_p+\mu_p)$
\item Summing over all possible such pairs of sequences $\{ \nu_1, \cdots, \nu_p \}$, $\{ \mu_1, \cdots, \mu_p \}$. 
\end{enumerate}

\begin{example}
We have
\begin{align*}
 Z(\alpha_1) Z(\beta_1, \beta_2) & = Z(\alpha_1 + \beta_1, \beta_2) + Z(\beta_1, \alpha_1 + \beta_2) + Z(\beta_1, \beta_2, \alpha_1) \\ &  + Z(\beta_1, \alpha_1, \beta_2) + Z(\alpha_1, \beta_1, \beta_2) \\
 \end{align*}
\end{example}
One checks readily that the two structures are compatible, and that they respect the $K^+_0 (\CRF_S)$--grading determined by  $$deg(Z(\alpha_1, \cdots, \alpha_k)) = \alpha_1 + \cdots + \alpha_k.$$
The pairing $$<, >: \Q_{\CRF_S} \times \NC_{\CRF_S} \rightarrow \mathbb{Q}$$ determined by
\[
< Z(\alpha_1, \cdots, \alpha_n), X_{\beta_1} \cdots X_{\beta_m} > := \delta_{m,n} \delta_{\alpha_1, \beta_1} \cdots \delta_{\alpha_n} \delta_{\beta_m}
\]
makes $\Q_{\CRF_S}$ and $\NC_{\CRF_S}$ into a \emph{dual pair} of $K^+_0 (\CRF_S)$--graded Hopf algebras. I.e. 
\begin{align*}
< a \otimes b, \Delta(v) > & = <ab, v> \\
<\Delta(a), v \otimes w> & = < a, vw> 
\end{align*}
This implies that $\Q_{\CRF_S}$ is isomorphic to the graded dual of $\NC_{\CRF_S}$. 

Passing to graded duals, and taking the transpose of the homomorphism $\rho$, we obtain a Hopf algebra homomorphism 
\[
\rho^t : \H^*_{\CRF_S} \rightarrow \Q_{\CRF_S}
\]
As shown in \cite{KS}, $\H^*_{\CRF_S}$ is isomorphic to the Connes-Kreimer Hopf algebra on colored trees (see \cite{K}). 

We proceed to describe $\rho^t$. Let $\{ W_{A}, A \in \on{Iso}(\CRF_S) \}$ be the basis of $\on{H}^*_{\CRF_S}$ dual to the basis $\{ \delta_{A} \}$ of $\H_{\CRF_S}$. 

\begin{theorem}
\[
\rho^t (W_A) = \sum_k \sum_{ V_1 \subset \cdots \subset V_k = A} Z([V_1], [V_2/V_1], \cdots, [V_k/V_{k-1}]) 
\]
where the inner sum is over distinct $k$--step flags $$V_1 \subset V_2 \subset \cdots \subset V_k = A, \; V_i \in \on{Iso}(\CRF_S).$$ 
\end{theorem}

\begin{proof}
We have
\[
\rho^t (W_A) (X_{\alpha_1} \cdots X_{\alpha_k}) = N(A; \alpha_1, \cdots, \alpha_k)  
\]
where $ N(A; \alpha_1, \cdots, \alpha_k)$ is the coefficient of $\delta_A$ in the product $\kappa_{\alpha_1} \kappa_{\alpha_2} \cdots \kappa_{\alpha_k}$. It follows from the definition of the multiplication in the Ringel-Hall algebra that this is exactly the number of flags 
$$V_1 \subset V_2 \cdots \subset V_k $$ where $[V_1]=\alpha_1, [V_2/V_1] = \alpha_2 ,\cdots,  [V_k/V_{k-1}] = \alpha_k$. 
\end{proof}

\begin{example} Let $S=\{a,b\}$ as in example \ref{ex1}. Using the notation introduced there, we have
\begin{align*}  \psset{levelsep=2ex, treesep=0.6cm, treenodesize=1pt,nodesepB=-2pt}
\rho^t \left( W_{  \pstree{\Tc*{3pt}~[tnpos=r]{a}}{\Tc*{3pt}~[tnpos=b]{b} \Tc*{3pt}~[tnpos=b]{a} }} \right) & = Z((2,1)) + Z((0,1), (2,0)) + Z((1,0), (1,1)) \\ & + Z((1,1), (1,0))  + Z((1,0),(0,1), (1,0)) + Z((0,1), (1,0), (1,0))
\end{align*}
\end{example}

\end{document}